\documentclass[11pt,longbibliography]{article}

\usepackage[right=3cm,left=3cm,textheight=22cm]{geometry}

\usepackage{amsmath}
\usepackage{amsthm}
\usepackage{amssymb}
\usepackage{amsfonts}
\usepackage{paralist}
\usepackage{url}
\usepackage{color}

\allowdisplaybreaks

\usepackage[colorlinks=true,linkcolor=blue,citecolor=blue,urlcolor=blue,breaklinks]{hyperref}
\hypersetup{pdftitle={Existence of Compactly Supported Global Minimisers for the Interaction Energy}}
\hypersetup{pdfauthor={José A. Cañizo, José A. Carrillo \& Francesco Saverio Patacchini}}

\numberwithin{equation}{section}

\theoremstyle{plain}
\newtheorem{thm}{Theorem}[section]
\newtheorem{cor}[thm]{Corollary}
\newtheorem{lem}[thm]{Lemma}
\newtheorem{prop}[thm]{Proposition}
\newtheorem{hyp}{Hypothesis}

\theoremstyle{definition}
\newtheorem{defn}[thm]{Definition}

\theoremstyle{remark}
\newtheorem{rem}[thm]{Remark}

\theoremstyle{example}

\newcommand{\be}{\begin{equation}}
\newcommand{\ee}{\end{equation}}
\newcommand{\bfig}{\begin{figure}}
\newcommand{\efig}{\end{figure}}
\newcommand{\bt}{\begin{table}}
\newcommand{\et}{\end{table}}
\newcommand{\bc}{\begin{center}}
\newcommand{\ec}{\end{center}}
\newcommand{\ba}{\begin{array}}
\newcommand{\ea}{\end{array}}
\newcommand{\bes}{\begin{equation*}}
\newcommand{\ees}{\end{equation*}}

\newcommand{\norm}[1]{\left\Vert#1\right\Vert}

\def\R{\mathbb{R}}
\def\N{\mathbb{N}}
\def\P{\mathcal{P}}

\def\div{\nabla\cdot}
\def\d{\,\mathrm{d}}

\def\ird{\int_{\R^d}}
\def\irdrd{\int_{\R^d\times\R^d}}
\def\st{\, | \,}

\def\:{\colon}
\def\rst{\big\vert}

\renewcommand{\L}[1]{L^{#1}(\R^d)}

\DeclareMathOperator{\supp}{supp}

\DeclareMathOperator{\diam}{diam}

\usepackage{ulem}
\begin{document}

\title{Existence of Compactly Supported Global Minimisers \\ for the Interaction Energy\footnote{(DOI) 10.1007/s00205-015-0852-3.}}
\author{Jos\'{e}
  A. Ca\~{n}izo\thanks{\url{j.a.canizo@bham.ac.uk}.
    School of Mathematics, Watson Building, University of Birmingham,
    Edgbaston, Birmingham B15 2TT, UK.}
  \and
  Jos\'e A. Carrillo\thanks{\url{carrillo@imperial.ac.uk}.
    Department of Mathematics, Imperial College London, South
    Kensington Campus, London SW7 2AZ, UK.}
  \and
  Francesco S. Patacchini\thanks{\url{f.patacchini13@imperial.ac.uk}.
    Department of Mathematics, Imperial College London, South
    Kensington Campus, London SW7 2AZ, UK.}}

\date{February 2015}

\maketitle


\begin{abstract}
  The existence of compactly supported global minimisers for continuum
  models of particles interacting through a potential is shown under
  almost optimal hypotheses. The main assumption on the potential is
  that it is catastrophic, or not H-stable, which is the complementary
  assumption to that in classical results on thermodynamic limits in
  statistical mechanics. The proof is based on a uniform control on
  the local mass around each point of the support of a global
  minimiser, together with an estimate on the size of the
  ``gaps'' it may have. The class of potentials
  for which we prove existence of global minimisers includes power-law
  potentials and, for some range of parameters, Morse potentials,
  widely used in applications. We also show that the support of local
  minimisers is compact under suitable assumptions.
\end{abstract}

\tableofcontents

\section{Introduction and main results}
\label{sec:introduction}

The analysis of configurations minimising nonlocal interaction
  energies is an ubiquitous question in mathematics with applications
  ranging from physics and engineering problems to mathematical
  biology and game theory in economic and social
  sciences. Understanding the balance of the effects of interactions
  between the ``particles'' in applications such as inelastic
  particles in granular flows
  \cite{BenedettoCagliotiCarrilloPulvirenti98,LT,CMV,CMV2}, molecules
  in self-assembly materials and virus structures
  \cite{Wales1995,Wales2010,Rechtsman2010,Viral_Capside}, animals in
  flock patterns in biological swarms \cite{BUKB,soccerball,CHM,
    CHM2}, and individuals' strategies in pedestrian dynamics or
  strategic preferences \cite{BC,DLR}, is of paramount importance.

  We deal with \emph{interaction potentials}: $W \: \R^d \to \R \cup
  \{+\infty\}$ is a pointwise defined and measurable function which is
  allowed to take the value $+\infty$, and whose gradient models the
  interaction force between two particles located at a distance
  $x\in\R^d$. More precisely, we regard $-\nabla W(y-x)$ as the force
  that a particle at $x$ exerts on a particle at $y$, and accordingly
  we say that $W$ is \emph{attractive} at $x \in \R^d$ when $-\nabla
  W(x) \cdot x \leq 0$, and \emph{repulsive} when $-\nabla W(x) \cdot
  x \geq 0$.  Given $N$ particles with positions $x_i\in\R^d$, for any
  $i \in \{1,\dots,N\}$, we can define the energy associated to them
  as
  $$
  E_N(x_1,\dots,x_N)=\frac1{2N^2} \sum_{i,j=1}^N W(x_i-x_j).
  $$
  The typical potentials that we have in mind are repulsive at short
  distances (i.e., for $|x| < r$ for some $r > 0$) and attractive at
  large ones, and it is the interplay of these two effects that allows
  for the existence of minimisers with interesting properties. The
  minimisers of the discrete energy should realise the most stable
  balance between the possible attractive and repulsive effects
  encoded in the interaction potential $W$. The normalisation of the
  discrete energy is done in such a way that it is kept of order one
  as $N\to \infty$. Finding global (and local) minimisers of this
  discrete energy $E_N$ is a question of major interest in
  crystallisation, where self-interaction is avoided, i.e., $W(0)=0$
  (see \cite{theil} and the references therein), but also for less
  singular potentials where the normalised discrete minima may
  converge towards some integrable density or non-atomic probability
  measure when $N\to \infty$.

  It is therefore more suitable to relax the variational problem and
  look for global minimisers of the \emph{interaction energy functional} $E \: \P(\R^d)
  \to \R \cup \{+\infty\}$, $d \geq 1$, defined on the set $\P(\R^d)$
  of probability measures on $\R^d$ by
  \begin{equation}
    \label{eq:energy}
    E(\rho) = \dfrac{1}{2} \irdrd W(x - y) \d\rho(x) \d\rho(y),
    \qquad \rho \in \P(\R^d).
  \end{equation}
  If one considers $\rho$ as a given mass distribution, then
  \eqref{eq:energy} is its total potential energy when individual
  particles interact through a pair potential $W$ (i.e., the potential
  energy of two particles with unit mass, one at $x \in \R^d$ and one
  at $y \in \R^d$, is $W(x-y)$). More precisely in the following, if
  $\rho \in \P(\R^d)$ satisfies $E(\rho) \leq E(\mu)$ for all $\mu \in
  \P(\R^d)$ we say that $\rho$ is a \emph{ground state} or
  \emph{global minimiser} of $E$ (and sometimes simply
  \emph{minimiser}). Analogously, we talk about minimisers on a subset
  $\mathcal{A} \subseteq \P(\R^d)$ when the inequality holds in
  $\mathcal{A}$.

  Let us mention that the set of global (and local) minimisers of the
  total potential energy for not too singular potentials can be very
  rich in terms of their qualitative properties depending on the
  behaviour of the potential at the origin. Actually, there are plenty
  of works reporting on the qualitative properties of critical points
  of the discrete energy in the context of collective behaviour (see
  \cite{DCBC06,BUKB,soccerball} and the references therein). For
  instance, Morse potentials were considered in \cite{DCBC06}, that
  is, potentials of the form
  \begin{equation}\label{mmm}
    W(x) = C_Re^{-\frac{|x|}{\ell_R}} - C_Ae^{-\frac{|x|}{\ell_A}}, \qquad x \in \R^d,
  \end{equation}
  with $C_R$, $C_A$ measuring the strengths of the repulsive and the
  attractive part, respectively, and $\ell_R$, $\ell_A$ being the
  typical lengths scales for repulsion and attraction,
  respectively. The authors' detailed numerical study indicates that
  these potentials lead to patterns corresponding to minimisers of the
  energy as $N\to \infty$ only in the range of parameters
  corresponding to $\ell_R < \ell_A$ and $C_R/C_A <
  \left( \ell_A / \ell_R \right)^d$. Furthermore, they noted that
  these conditions are intimately related to the classical notion of
  $H$-stability in statistical mechanics.

  Before discussing further this connection, we now introduce the set
  of hypotheses on the interaction potential needed to state our main
  results in full rigour. We always assume, without loss of
  generality, that
  \begin{hyp}
    $W$ is bounded from below by a finite constant $W_\mathrm{min} < 0$.
    \label{hyp:bfb}
  \end{hyp}
  This ensures that \eqref{eq:energy} is well defined (possibly equal
  to $+\infty$) for all $\rho \in \P(\R^d)$. Actually, in order for
  $E$ to be finite on ``nice'' probability measures $\rho$ (say,
  uniformly distributed on a ball), we also assume
  \begin{hyp}
    $W$ is locally integrable (that is, $\int_B |W| < +\infty$ for any open ball $B \subset \R^d$).
    \label{hyp:integrable}
  \end{hyp}
  In particular, Hypothesis \ref{hyp:integrable} implies that the
  potential $W$ cannot be equal to $+\infty$ on sets with positive
  Lebesgue measure. This assumption rules out interaction potentials
  which are too singular at the origin as those used in the analysis
  of the crystallisation phenomena \cite{theil,eli} (see further
  comments before Hypothesis \ref{hyp:instability} below).

Since we think of $W$ as a potential as explained above, 
it is natural to assume additionally that
\begin{hyp}
  $W$ is symmetric (that is, $W(x) = W(-x)$ for all $x \in \R^d$).
\label{hyp:sym}
\end{hyp}
It is also quite natural to assume that the potential $W$ is radial, but
since we do not need that in the following we avoid making the
assumption. 

The problem of finding global minimisers of \eqref{eq:energy} has two
fundamental invariances. First, $E$ is invariant under translations:
for all $\rho \in \P(\R^d)$ and $z \in \R^d$ we have $E(\rho) = E(T_z(\rho))$, where $T_z(\rho)$ is the push-forward of $\rho$ by the $z$-translation $T_z \: \R^d \to
  \R^d$, $x \mapsto x - z$, that is $T_z(\rho)(x)=\rho(x+z)$ for integrable densities. In
particular, any translation of a minimiser is also a minimiser, and
uniqueness (if it holds) can only be expected up to
translations. Second, if we add an arbitrary constant $D \in \R$ to
$W$, then the energy $E$ is shifted by $\frac 12 D$, and hence the
minimisation problem does not change (note that this is consistent
with the interpretation of $W$ as a potential, which is arbitrary up
to a constant).

Our main motivation for studying these minimisers has mainly come from the
recent interest in the field of collective behaviour regarding the
steady states of the aggregation equation
\begin{equation}
  \label{eq:aggregation}
  \partial_t \rho = \div (\rho (\nabla W * \rho)),
\end{equation}
where $\rho = \rho(t,x)$ is a function (possibly a measure) defined
for $t \geq 0$ and $x \in \R^d$. Since $E$ is a Lyapunov functional
for Equation \eqref{eq:aggregation} (in fact, \eqref{eq:aggregation}
is the gradient flow of $E$ with respect to the Wasserstein transport
distance \cite{CMV,CMV2,CDFLS}) its minimisers (if they exist) are natural
candidates for a steady state of \eqref{eq:aggregation}, and they are
also natural candidates to represent the typical asymptotic behaviour
of \eqref{eq:aggregation} as $t \to +\infty$. Equations of the form \eqref{eq:aggregation}
appear in granular flow \cite{LT,CMV,CMV2} and in swarming models (see \cite{BCL,BT2,KCBFL} and the references therein) as mentioned before.

It is easy to see that if a minimiser $\rho$ happens to be regular
enough then it must satisfy the corresponding \emph{Euler-Lagrange equation},
namely
\begin{equation}
  \label{eq:Euler-Lagrange}
  W * \rho = C
  \quad \text{on $\supp\rho$},
\end{equation}
for some $C \in \R$ (see Lemma \ref{lem:constant} and Remark
\ref{rem:constant-general} for a more precise
statement). Consequently, it must be a stationary state of
\eqref{eq:aggregation}, again assuming that $\rho$ is regular enough
for the right-hand side of \eqref{eq:aggregation} to be meaningful.
Equation \eqref{eq:Euler-Lagrange} also appears as a condition
satisfied by special solutions in a variety of models (for example,
flock solutions in \cite{DCBC06,CHM,CHM2,Albietal}), and it is
interesting in itself.

As we already mentioned, the authors in \cite{DCBC06} analysed numerically a discrete collective behaviour model based on the Morse potentials of the form \eqref{mmm},
and it was noticed that its large-time asymptotics seemed to depend on whether the potential satisfied a
classical condition known as \emph{stability} or \emph{H-stability} in
classical statistical mechanics \cite{FisherRuelle66,Ruelle}:
\begin{defn}[Stability]
  Take a potential $W \: \R^d \to \R \cup \{+\infty\}$ satisfying
  Hypotheses \ref{hyp:bfb} and \ref{hyp:integrable}, and assume that
  its limit at $\infty$ exists (being possibly equal to
  $+\infty$). Let us define
  \begin{equation}
    W_\infty :=
    \lim_{|x| \to \infty} W(x).
    \label{eq:Einf}
  \end{equation}
  We say that $W$ is \emph{stable} if
  \begin{equation}
    \label{eq:stability}
    E(\rho) \geq \frac12 W_\infty
    \quad \text{for all $\rho \in \P(\R^d)$.}
  \end{equation}
  Otherwise we say that $W$ is \emph{unstable}. In a similar way, we define the concept of
  stability/instability on a subset $\mathcal{A} \subseteq \P(\R^d)$
  by restricting \eqref{eq:stability} to all $\rho \in \mathcal{A}$.
  \label{defn:stability}
\end{defn}
The term \emph{catastrophic} instead of unstable is also used in part
of the statistical mechanics literature, where a common assumption is
that $W(x) \to 0$ as $|x| \to \infty$; in that case
eq. \eqref{eq:stability} translates to $E(\rho) \geq 0$ for all $\rho
\in \P(\R^d)$, which is the definition often given classically (see
for example \cite[eq. (23)]{Bavaud}). The notion of stability given in
Definition \ref{defn:stability} implies the classical notion of
stability if the potential is such that $W(0)$ is finite (see for
example \cite[Sections 3.1 and 3.2]{Ruelle}). In statistical mechanics
the classical condition of stability is motivated by the fact that, when combined with that of temperedness, it is sufficient to show the thermodynamic behaviour of
a system of particles interacting via a pairwise potential (see \cite[Theorem
  3.3.12]{Ruelle}). Indeed,
the former condition avoids that infinitely many particles collapse in a bounded
region and the latter avoids significant interaction between distant
particles.

Last but not least, the crystallisation phenomena discussed in
\cite{theil,eli} (see \cite{FM,YFS} for other related work) are
closely linked to the problem we consider, though the energy there is
minimised among configurations of a finite number of particles, and
does not include their self-interaction. The results in these
references give more detailed information about the behaviour of the
minimisers for large number of particles in particular stable
cases. Essentially, the minimisers form regular lattices which spread
and fill the whole space as the number of particles tends to
infinity. In other words, if we insist in normalising the minimisers
as probability measures, these Dirac delta minimisers tend to zero
weakly as measures as the number of particles tends to infinity. In
fact, the infimum of the interaction energy (modified as mentioned
above to exclude self-interactions) over linear combinations of Dirac
deltas is zero but is not attained.

By contrast, in our setting we need to assume that $W$ is unstable in
order to show the existence of a minimiser:
\begin{hyp}
  \label{hyp:instability}
  The limit \eqref{eq:Einf} exists and $W$ is unstable.
\end{hyp}
We also assume the following more technical hypotheses,
though some of our results do not require Hypothesis
  \ref{hyp:growth}:
\begin{hyp}
  $W$ is lower semi-continuous.
\label{hyp:lsc}
\end{hyp}

\begin{hyp}
  There is $R_6 > 0$ such that $W$ is strictly increasing on
  $\R^{k-1} \times [R_6,\infty) \times \R^{d-k}$ as a function of its
  $k$-th variable, for all $k \in \{1,\dots,d\}$.
  \label{hyp:growth}
\end{hyp}

\begin{rem}
  Notice that any potential which is radially strictly increasing
  outside some ball (i.e., for some $R_6 > 0$ we have $W(x) > W(y)$
  whenever $|x| > |y| \geq R_6$) satisfies Hypothesis \ref{hyp:growth}.
\end{rem}

\begin{rem}
  Hypothesis \ref{hyp:growth} is unnecessary in order to show
  existence of global minimisers (see the related results in
  \cite{SST}). However, some growth condition such as Hypothesis
  \ref{hyp:growth} seems to be necessary to show the uniform compactness part
  of Theorem \ref{thm:main} below (see Remark \ref{rem:uniform}). An example of growth condition which we could use here that is less restrictive (but also less intuitive) than Hypothesis \ref{hyp:growth} is given in Remark \ref{rem:growth-relaxation}.
\end{rem}
\medskip

Our main result is the following:

\begin{thm}
  \label{thm:main}
  Assume that $W\:\R^d \to \R \cup \{+\infty\}$ satisfies Hypotheses
  \ref{hyp:bfb}--\ref{hyp:growth}. Then there exists a global
  minimiser $\rho \in \P(\R^d)$ of the energy $E$. In addition, there
  exists $K > 0$ (depending only on $W$ and the dimension $d$) such
  that every minimiser of $E$ has compact support with diameter at
  most $K$.
\end{thm}

We point out that an explicit estimate on the size of the support can
be recovered from the proof of the above theorem in Section
\ref{sec:existence}, though we do not expect it to be sharp.

Our proof of Theorem \ref{thm:main} rests on two key apriori
  estimates on minimisers. First, in Lemma
  \ref{lem:no-small-isolated-mass} we show that any point in the
  support of a minimiser needs to have at least a fixed amount $m > 0$
  of mass which is not further away than a fixed distance $r >
  0$. The intuitive reason for this is that the potential energy
  $W*\rho$ has to be constant at $\rho$-almost every point of the support of
  a minimiser $\rho$ (cf. Lemma \ref{lem:constant} and \cite[Theorem
  4(i)]{BCLR2}), and that potential energy has to be strictly less
  than the ``potential at infinity'' $W_\infty$ due to the
  instability assumption in Hypothesis \ref{hyp:instability}. Hence some
  mass has to be close to any given point $x$ in the support, since mass
  being too far away would mean that $W * \rho(x)$ is too close
  to $W_\infty$. Secondly, Lemma \ref{lem:no-large-holes} uses Hypothesis
  \ref{hyp:growth} to show that the support of a minimiser cannot have
  arbitrarily large ``gaps''. Otherwise one could bring closer the two
  parts of the minimiser separated by the gap and obtain a mass
  distribution with a smaller energy. This, together with the first estimate, shows that a minimiser has to consist of at most
  $\lceil 1/m \rceil$ pieces, each with mass at least $m$, not too far
  apart from each other, where $\lceil\cdot\rceil$ is the ceiling function. An apriori estimate on the size of the
  support is then easily obtained. 

The proof of Theorem \ref{thm:main} is then completed by a usual
  approximation argument: we consider minimisers $\rho_R$ among the set of
  measures supported in $\overline{B}(0,R)$, the closed ball of centre 0 and radius some $R$, and show that these estimates hold
  uniformly for them, which allows us to pass to the limit as $R \to
  +\infty$ to obtain a global minimiser.

We have recently learnt that Simione, Slep\v{c}ev \& Topaloglu
\cite{SST} independently proved a similar result by a different
method based on Lions' concentration compactness principle while this
paper was being prepared. Their method does
not give any estimate on the support or properties of the minimisers; on the other hand, their conditions for existence of minimisers are
slightly sharper than the ones in Theorem \ref{thm:main}. To complement their result we give a corollary which derives directly from the structure of our proof of Theorem \ref{thm:main}. It is proved in Section \ref{subsec:all-compact-support} below.
\begin{cor}
	Assume that $W\:\R^d \to \R \cup \{+\infty\}$ satisfies Hypotheses
  \ref{hyp:bfb}--\ref{hyp:lsc} (satisfied by any potential considered in \cite{SST} as long as it is unstable in our sense). Suppose moreover that there exists a global minimiser $\rho \in
  \P(\R^d)$ of the energy
  \eqref{eq:energy}. Then $\rho$ is compactly supported.

\label{cor:complement}
\end{cor}
\begin{rem}
  We emphasise that Corollary \ref{cor:complement} does not require
  Hypothesis \ref{hyp:growth} on the growth at $\infty$ of the
  potential $W$. Hence global minimisers for potentials which, for
  instance, decrease to 0 at $\infty$ and satisfy Hypotheses
  \ref{hyp:bfb} to \ref{hyp:lsc}, have to be compactly supported. Note that our
  statement in this case does not show the existence of any uniform bound $K$ on the size of the support of a global minimiser (unlike Theorem \ref{thm:main}).
\label{rem:uniform}
\end{rem}
\medskip

\emph{Local} minimisers of the energy \eqref{eq:energy} in several
transport distances (i.e., minimisers in a small ball around them)
were studied in \cite{BCLR2}, where bounds on the dimension of their
support were given under some assumptions controlling the strength of
the repulsion at the origin. Moreover, Euler-Lagrange conditions for
local minimisers in several transport distances were also
obtained. These results (see \cite[Theorem 4]{BCLR2}) apply to our
case since a global minimiser is of course a local one in particular.
Note that, in Section \ref{subsec:corollary-local-minimisers} below, we derive a generalisation of Corollary \ref{cor:complement} to $d_2$-local minimisers, where $d_2$ is the quadratic Wasserstein distance. Let us mention that the rich structure of global and local minimisers of
the interaction energy was shown for several potentials and by
different numerical methods in
\cite{soccerball,FHK,FH,BCLR,BCLR2,BCY,Albietal,CHM,CCH2}.

Let us review some of the known rigorous results on the existence, uniqueness,
and other properties of these minimisers in some remarkable cases. An
often studied case is that in which the potential $W$ is a sum of
powers:
\begin{equation*}
  W(x) = \frac{|x|^a}{a} - \frac{|x|^b}{b},
  \qquad x \in \R^d,
\end{equation*}
for some $a, b \in \R$ with $-d < b < a$, and the understanding that
$|x|^0/0 \equiv \log |x|$. Here the term $|x|^a / a$ is the attractive
one (being an increasing function of $|x|$, regardless of the sign of
$a$), and $|x|^b/b$ is the repulsive one (since it is a decreasing
function of $|x|$). For $b = 2-d$ the repulsive term is called the
\emph{Newtonian potential}. It is not difficult to check that this class of potentials satisfies our
hypotheses (see Section \ref{sec:examples-nonexistence}).

The case $a=2$ simplifies the problem a lot since the attractive part
of the interaction can be reduced to an external quadratic confinement
by expanding the square. The case $b = 2-d$, $a = 2$ is actually
relatively well-known among probabilists: up to translations, the
unique global minimiser is the characteristic of a ball with an
appropriate radius. A closely related result with a compact
confinement is proved in \cite{frostman}, and the 2-dimensional case
can be found for example in \cite[Theorem 6.1,
p. 245]{Saff-Totik}. The extension to higher dimensions can be found
in \cite[Proposition 2.13]{LG2010}. The interest in this problem on
these references comes from its links to the capacity of sets and
applications to random matrix theory (see \cite{Chafai} and the
references therein).

A modified minimisation problem for power-law potentials with $a
  > -d$ was recently studied in \cite{CFT}, where the authors showed
  that there exists a minimiser $\rho_M$ in the class of radially
  symmetric functions with a fixed $L^\infty$-bound $M$. They also
  show that for $-d < a < 0$ the condition of radial symmetry is not
  needed, so that in that case a minimiser exists in the class of functions
  with a given $L^\infty$-bound.

The case $a = 2$, $b = 2s - d$ for $0 < s < 1$ was studied in
\cite{Caffarelli2011Asymptotic} in relation to the asymptotic
behaviour of eq. \eqref{eq:aggregation}, referred to as the
\emph{fractional porous medium equation}. The authors there showed
that there is a unique steady state (up to translations) to
\eqref{eq:aggregation}, which they called a \emph{modified Barenblatt
  profile}. Since the uniqueness was shown via the associated obstacle
problem and the Euler-Lagrange conditions in \cite[Theorem 4]{BCLR2}
show that global minimisers are regular solutions of the obstacle
problem (see also \cite{CDM}), then their uniqueness result implies
uniqueness of global minimisers. Finally, all cases with $-d<b\leq
2-d$ and $a>0$ were recently treated in \cite{CDM} showing the
regularity of local minimisers using the connection to classical
obstacle problems, by methods that can treat more general potentials
than power laws: results in \cite{CDM} apply to potentials behaving
like $-|x|^b/b$ at zero in the range $-d<b\leq 2-d$ with a smooth
enough attractive part of the potential.

On the other hand, there have been many works devoted to the study of
the steady states and long-time behaviour of
eq. \eqref{eq:aggregation} (see \cite{FellnerRaoul1,FellnerRaoul2,Raoul,FHK,FH,BCLR,BCLR2,CFP,CDFLS,BCY} and the references therein). Steady states for the case $b = 2-d$, $a > 0$ were studied in \cite{FHK,FH}, where it was
proved that there exists a unique radial compactly supported steady state (up
to translations). The asymptotic behaviour of
\eqref{eq:aggregation} in the case $b = 2-d$, $a = 2$ was studied
in \cite{FHK,BLL12}, and the case $-d < b < 2 - d$, $a = 2$, as
already remarked, was considered in
\cite{CV11,Caffarelli2011Asymptotic}. 

Finally, other particular interesting potentials are Morse-like
potentials \cite{DCBC06,BT2,BCY,CMP,CHM2} (treated in Section \ref{sec:examples-nonexistence}), for which there is a huge
numerical evidence of the existence of compactly supported global
minimisers. To our knowledge, a proof of this existence was not previously
available.
\medskip

The paper is organised as follows: in Section \ref{sec:existence} we
prove Theorem \ref{thm:main}. We split the proof of the existence part
into three main steps: first we show existence of global minimisers on
each set of probability measures supported on a given ball, second we
prove that the support diameters of such minimisers are uniformly
bounded, and third we show existence on the whole space $\P(\R^d)$. At the end of this section we give the generalisation of Corollary \ref{cor:complement} to local minimisers with respect to the quadratic Wasserstein topology. In
Section \ref{sec:examples-nonexistence} we give examples of potentials
satisfying the hypotheses of Theorem \ref{thm:main}, as well as
conditions for the non-existence of global minimisers which show that
Hypothesis \ref{hyp:instability} in Theorem \ref{thm:main} is almost sharp. 

\section{Existence of minimisers}
\label{sec:existence}

\subsection{Minimisers on a given ball}
\label{subsec:compactly-supported-each-set}

Let us define for all $R \geq 0$ the set $\P_R(\R^d) = \left\{ \rho \in \P(\R^d) \st \supp\rho \subset \overline{B}(0,R)\right\}$. This is the set of all probability measures with support included in
the closed ball $\overline{B}(0,R)$. For every $R \geq 0$, we want to
show existence of global minimisers on $\P_R(\R^d)$.

\medskip

Let us first show the following lemma:
\begin{lem}
  Let $W$ be a potential satisfying Hypotheses \ref{hyp:bfb} and
  \ref{hyp:lsc}. Then the energy \eqref{eq:energy} is weakly-$\star$
  lower semi-continuous, i.e., for any sequence $(\rho_n)_{n\in\N}$
  converging weakly-$\star$ to $\rho$ in $\P(\R^d)$ we have $E(\rho)
  \leq \liminf_{n\to\infty} E(\rho_n)$.
\label{lem:weak-lsc}
\end{lem}

\begin{proof}
  By Hypotheses \ref{hyp:bfb} and \ref{hyp:lsc} we know that there
  exists a non-decreasing sequence $(\varphi_m)_{m\in\N}$ of
  continuous and bounded functions such that $\varphi_m \to W$
  pointwise as $m \to \infty$ (see \cite[Lemma A.1.3]{Bauerle}). Let
  us consider a bound from below for $\varphi_1$, and denote it by
  $c$. Then $(\varphi_m - c)_{m\in\N}$ is non-negative, non-decreasing
  and with pointwise limit $W - c$. Suppose that $(\rho_n)_{n\in\N}$
  is a sequence weakly-$\star$ converging to $\rho$ in
  $\P(\R^d)$. Then, applying the Lebesgue monotone convergence
  theorem, we obtain
\begin{equation*}
  \irdrd (\varphi_m(x-y) - c) \d\rho(x) \d\rho(y) \xrightarrow[m\to\infty]{} \irdrd (W(x-y) - c) \d\rho(x) \d\rho(y).
\end{equation*}
Therefore, we infer
\begin{equation*}
  \irdrd \varphi_m(x-y) \d\rho(x) \d\rho(y) \xrightarrow[m\to\infty]{} \irdrd W(x-y) \d\rho(x) \d\rho(y).
\end{equation*}
Furthermore, since $(\varphi_m)_{m\in\N}$ is non-decreasing, we have
for all $n,m \in \N$,
\begin{equation*}
  \irdrd \varphi_m(x-y) \d\rho_n(x) \d\rho_n(y)
  \leq
  \irdrd W(x-y) \d\rho_n(x) \d\rho_n(y).
\end{equation*}
Hence, by definition of weak-$\star$ convergence and
passing to the limits $n,m \to \infty$, we get
\begin{equation*}
  \irdrd W(x-y) \d\rho(x) \d\rho(y)
  \leq \liminf_{n\to\infty} \irdrd W(x-y) \d\rho_n(x) \d\rho_n(y),
\end{equation*}
proving the desired result.
\end{proof}

We can now state an existence result for global minimisers of the
energy \eqref{eq:energy} on $\P_R(\R^d)$ for any $R \geq 0$.
\begin{lem}
  Suppose that the potential $W$ satisfies Hypotheses \ref{hyp:bfb}
  and \ref{hyp:lsc}. Then for every $R \geq 0$ there exists a global
  minimiser $\rho_R$ on $\P_R(\R^d)$ of the energy \eqref{eq:energy}.
\label{lem:existence-r}
\end{lem}

\begin{proof}
  Let us fix $R \geq 0$ and let $(\rho_n)_{n\in\N}$ be a minimising
  sequence for the restriction of the energy \eqref{eq:energy} to
  $\P_R(\R^d)$. Note that $\P_R(\R^d)$ is a tight subset of
  $\P(\R^d)$, and therefore by Prohorov's theorem we know there
  exists a subsequence $(\rho_{n_k})_{k\in\N}$ and $\rho_R$ such that
  $\rho_{n_k} \rightharpoonup \rho_R$ weakly-$\star$ as $k \to
  \infty$. Moreover, since $\P_R(\R^d)$ is weakly-$\star$ closed, we have
  $\rho_R \in \P_R(\R^d)$. By Lemma \ref{lem:weak-lsc} we have
  $\liminf_{k \to \infty} E(\rho_{n_k}) \geq E(\rho_R)$. Hence,
  $\inf\{E(\rho) \st \rho \in \P_R(\R^d)\} = \lim_{n \to \infty}
  E(\rho_n) = \liminf_{k \to \infty} E(\rho_{n_k}) \geq E(\rho_R) \geq
  \inf\{E(\rho) \st \rho \in \P_R(\R^d)\}$. Therefore $\rho_R$ is a
  global minimiser of the energy \eqref{eq:energy} on $\P_R(\R^d)$.
\end{proof}

\subsection{Uniform bound on the support of minimisers}
\label{subsec:uniform-bound-minimisers}

We now show the existence of a bound for the diameter of the support
of global minimisers in $\P_R(\R^d)$ for any $R \geq 0$, which is
uniform in $R$.
\medskip

In \cite[Theorem 4(i)]{BCLR2} it was proved that a global minimiser
$\rho$ satisfies $W \ast \rho = 2E(\rho)$ $\rho$-almost everywhere. We
adapt it here for minimisers on $\P_R(\R^d)$ for any $R \geq 0$:
\begin{lem}
  Assume that the potential $W$ satisfies Hypotheses
  \ref{hyp:bfb}--\ref{hyp:sym} and \ref{hyp:lsc}. Take $R \geq 0$ and
  let $\rho_R$ be a global minimiser on $\P_R(\R^d)$. Then $W \ast
  \rho_R = 2E(\rho_R)$ $\rho_R$-almost everywhere.
\label{lem:constant}
\end{lem}

\begin{proof}
  Let $\varphi \in C_0^\infty(\R^d)$ and define
  \begin{equation*}
    \mathrm{d}\nu(x) = \left(\varphi(x) - \int_{\R^d} \varphi(y) \d\rho_R(y)\right) \mathrm{d}\rho_R(x)
  \end{equation*}
  for all $x \in \R^d$. Also consider $\rho_\varepsilon = \rho_R + \varepsilon\nu$ for
  $\varepsilon > 0$. Then we have
  \begin{equation*} 
    \int_{\R^d} \d\rho_\varepsilon(x) = \int_{\R^d} \d\rho_R(x) + \varepsilon\int_{\R^d} \d\nu(x) = 1.
  \end{equation*} 
  Moreover, one can check that
  \begin{equation*}
    \varphi(x) - \int_{\R^d} \varphi(y) \d\rho_R(y) \geq -2\norm{\varphi}_{\L{\infty}}.
  \end{equation*}
  Thus, for any Borel set $A$ of $\R^d$ we have
  \begin{equation*}
    \rho_\varepsilon(A) = \int_{A} \d\rho_\varepsilon(x) \geq \int_{A} \d\rho_R(x) - 2\varepsilon\norm{\varphi}_{\L{\infty}} \rho_R(A) = \left(1 - 2\varepsilon\norm{\varphi}_{\L{\infty}}\right)\rho_R(A).
  \end{equation*}
  Therefore, since $\rho_R$ is a probability measure, 
  \begin{equation*}
    \ba{ll}
    \rho_\varepsilon(A) \geq 0 & \mbox{if $\varepsilon \leq \dfrac{1}{2\norm{\varphi}_{\L{\infty}}}$}.
    \ea
  \end{equation*}
  Let us take such an $\varepsilon$, which ensures that
  $\rho_\varepsilon$ is a probability measure. Furthermore, we have
  $\supp \rho_\varepsilon \subset \overline{B}(0,R)$. Hence
  $\rho_\varepsilon \in \P_R(\R^d)$. We know that $\rho_R$ is a global
  minimiser on $\P_R(\R^d)$, and thus $E(\rho_\varepsilon) \geq
  E(\rho_R)$. In addition, by Hypotheses \ref{hyp:bfb} and
  \ref{hyp:integrable} the energy generated by $\rho_R$ is bounded. Then
  \begin{equation*}
    \dfrac{E(\rho_\varepsilon) - E(\rho_R)}{\varepsilon} = \int_{\R^d\times\R^d} W(x - y) \d\nu(x) \d\rho_R(y) + \dfrac{\varepsilon}{2}\int_{\R^d\times\R^d} W(x - y) \d\nu(x) \d\nu(y) \geq 0.
  \end{equation*}
  Hence, letting $\varepsilon \to 0$ and since the last integral is
  finite, we get
  \begin{equation*}
    \int_{\R^d\times\R^d} W(x - y) \d\nu(x) \d\rho_R(y) \geq 0,
  \end{equation*}
  or equivalently, by plugging the definition of $\nu$ inside the integral,
  \begin{equation*}
    \int_{\R^d} \left(W \ast \rho_R(x) - 2E(\rho_R)\right)\varphi(x) \d\rho_R(x) \geq 0.
  \end{equation*}
  This result being true for all $\varphi \in C_0^\infty(\R^d)$, we have
  that $W \ast \rho_R(x) - 2E(\rho_R) = 0$ $\rho_R$-almost everywhere.
\end{proof}

\begin{rem}
  By following the same argument one can see that Lemma
  \ref{lem:constant} is also true for a global minimiser $\rho$ on the
  whole space $\P(\R^d)$ (this is the content of \cite[Theorem 4(i)]{BCLR2}).
\label{rem:constant-general}
\end{rem}

To show that, under the instability condition in Hypothesis
\ref{hyp:instability}, the diameter of the support of a global
minimiser $\rho_R$ on $\P_R(\R^d)$ is independent of $R$, we first
notice that an unstable potential is also unstable on $\P_{{S}}(\R^d)$
for some finite radius ${S}$:

\begin{lem}
  \label{lem:instability-finite-R}
  Assume Hypotheses \ref{hyp:bfb} and \ref{hyp:instability} for the potential $W$. Then
  there exists ${S} > 0$ such that $W$ is unstable on $\P_{{S}}(\R^d)$
  (and hence on $\P_R(\R^d)$ for all $R \geq {S}$).
\end{lem}

\begin{proof}
  Let $\rho \in \P(\R^d)$ be such that $E(\rho) < \frac12 W_\infty$. Define the
  sequence of truncated probabilities $(\rho_n)_n \subset \P(\R^d)$,
  for every $n \in \N$ large enough so that $\rho(B(0,n)) > 0$, by
  $$\rho_n = \frac{1}{\rho(B(0,n))} \chi_{B(0,n)}\rho.$$ (Where
  $\chi_A$ denotes the characteristic function of a set $A$.) Clearly,
  for every such $n$, we have $\rho_n \in \P_n(\R^d)$. It is easy to
  see that $E(\rho_n) \to E(\rho)$ as $n \to \infty$, and hence there
  exists $N \in \N$ large enough such that $E(\rho_N) < \frac12 W_\infty$ (see
  the proof of Lemma \ref{lem:non-compactly-supported-measures} for a
  similar calculation). This proves the lemma with ${S} = N$.
\end{proof}

Below we always consider ${S} > 0$ to be a radius obtained from Lemma
\ref{lem:instability-finite-R}; that is, a number such that $W$ is
unstable on $\P_{{S}}(\R^d)$.

The following two lemmas are fundamental in the proof of our main
  result. The first one shows that if, for some $R$, a point is in the support of a
  minimiser $\rho_R$ on $\P_R(\R^d)$, then there has to be at least some mass not far
  from it. The quantification of ``some'' and ``not far'' are
  independent of $R$ and the point one chooses, so that this is a
  uniform estimate for all minimisers and all points:

\begin{lem}
  Suppose that the potential $W$ satisfies Hypotheses
  \ref{hyp:bfb}--\ref{hyp:lsc}. Then there are constants $r, m > 0$
  (depending only on $W$) such that for all $R \geq {S}$ and all
  global minimisers $\rho_R$ of the energy \eqref{eq:energy} on
  $\P_R(\R^d)$ we have
\begin{equation*}
  \int_{B(x_0,r)} \d \rho_R(x) \geq m
  \quad
  \text{for all $x_0 \in \supp \rho_R$.}
\end{equation*}

\label{lem:no-small-isolated-mass} 
\end{lem}

\begin{proof}
  We proceed in two steps: we first prove the result $\rho_R$-almost
  everywhere, and then everywhere in $\supp \rho_R$.
  \medskip
  
  \underline{Step 1: $\rho_R$-almost everywhere.} Call $E_R$ the
  minimum of the energy on $\P_R(\R^d)$; that is, $E_R := \min \left\{
    E(\rho) \mid \rho \in \P_R(\R^d) \right\}$. (We know this minimum
  exists due to Lemma \ref{lem:existence-r}.)  Clearly $E_R$ is
  non-increasing in $R$ and $E_R \leq E_{{S}} < \frac12
    W_\infty$ for all $R \geq {S}$ by our choice of ${S}$ (see Lemma
  \ref{lem:instability-finite-R}). If we consider a global minimiser
  $\rho_R$ on $\P_R(\R^d)$, we know by Lemma \ref{lem:constant} that
  for $\rho_R$-almost all $z \in \supp \rho_R$ we have
  \begin{equation*}
    \frac 12 \ird W(z - x) \d \rho_R(x)
    = E(\rho_R) = E_R \leq E_{{S}} < \frac12 W_\infty.
  \end{equation*}
  Note that $E_{{S}}$ is independent of $R$ and of the choice of the
  global minimiser $\rho_R$. Choose $A \in \R$ with $E_{{S}} < A <
  \frac12 W_\infty$. Since by definition we have $\lim_{|x| \to
    \infty} W(x) = W_\infty$, we can choose $r' > 0$ with $W(x)
  \geq 2A$ for all $x \in \R^d$ such that $|x| \geq r'$. (Notice that
  both $A$ and $r'$ are independent of $R$.) Then for $\rho_R$-almost
  every $z$ we have
  \begin{align*}
    2 E_R &= \ird W(z - x) \d \rho_R(x) = \int_{B(z,r')} W(z - x) \d \rho_R(x)
    + \int_{\R^d \setminus B(z,r')} W(z - x) \d \rho_R(x)
    \\
    &\geq
    W_\mathrm{min} \int_{B(z,r')} \d \rho_R(x)
    + 2 A \int_{\R^d \setminus B(z,r')} \d \rho_R(x)
    =
    (W_\mathrm{min} - 2A) \int_{B(z,r')} \d \rho_R(x)
    + 2 A,
  \end{align*}
  where we have used that $\rho_R$ is a probability
  measure. Rearranging terms and noticing that $W_\mathrm{min} - 2A <
  W_\mathrm{min} - 2E_{{S}} < 0$ and $2E_{{S}} \geq 2E_R$,
  \begin{equation}
    \label{eq:no-small-isolated-mass}
    \int_{B(z,r')} \d \rho_R(x)
    \geq \frac{A - E_{{S}}}{A - \frac 12 W_{\mathrm{min}}}
    =: m.
  \end{equation}
  This finishes this step, since the right-hand side depends only on $W$.
  
  \medskip
  
  \underline{Step 2: everywhere.} Take $\delta > 0$, call $r := r' +
  \delta$ and let $x_0 \in \supp \rho_R$. Then we get
  $\rho_R(B(x_0,\delta)) > 0$ by definition of the support of
  $\rho_R$. Suppose first that $\rho_R(B(x_0,\delta) \setminus
  \{x_0\}) = 0$. Then $\rho_R(\{x_0\}) > 0$ and therefore
  \eqref{eq:no-small-isolated-mass} has to be satisfied at $x_0$, so
  \begin{equation*}
    \int_{B(x_0,r)} \d \rho_R(x) \geq \int_{B(x_0,r')} \d \rho_R(x) \geq m.
  \end{equation*}
  Suppose now that $\rho_R(B(x_0,\delta) \setminus \{x_0\}) > 0$. Then
  there is $y_0 \in B(x_0,\delta) \setminus \{x_0\}$ such that
  \eqref{eq:no-small-isolated-mass} has to be satisfied at
  $y_0$. Thus, since $B(y_0,r') \subset B(x_0,r)$,
  \begin{equation*}
    \int_{B(x_0,r)} \d \rho_R(x) \geq \int_{B(y_0,r')} \d \rho_R(x) \geq m.
  \end{equation*}
  Hence the result holds for all $x_0 \in \supp \rho_R$ as $r$ depends
  only on $W$.
\end{proof}

The following lemma is in some sense complementary to the
  previous one: we show that if a minimiser $\rho_R$ on $\P_R(\R^d)$ has a ``gap'' in
  one of the coordinates, then it cannot be very big, with the
  quantification of ``very big'' being again independent of $R$ and
  the position of the gap. It is interesting to note that this is the
  only part in the paper where Hypothesis \ref{hyp:growth} is
  explicitly used: all later dependence on this hypothesis is through
  this lemma.
  
  In order to state this precisely let us denote by $\pi_k\:\R^d \to
  \R$ the $k$-th coordinate projection, for $k \in
  \{1,\dots,d\}$. Then the following holds:

\begin{lem}
  Assume that the potential $W$ satisfies Hypotheses \ref{hyp:bfb},
  \ref{hyp:sym} and \ref{hyp:growth}. Let $R \geq 0$ and suppose that
  $\rho_R$ is a global minimiser of the energy \eqref{eq:energy} on
  $\P_R(\R^d)$. Then the support of $\rho_R$ cannot have ``gaps''
  larger than $2 R_6$ in each coordinate (where $R_6$ is
    the constant from Hypothesis \ref{hyp:growth}): if $k \in
    \{1,\dots,d\}$ and $a_k \in \R$ is such that
  $\pi_k^{-1}([a_k-R_6,a_k+R_6]) \subseteq \R^d
  \setminus \supp \rho_R$, then either
  $\pi_k^{-1}((-\infty,a_k-R_6]) \subseteq \R^d \setminus
  \supp \rho_R$ or $\pi_k^{-1}([a_k+R_6,\infty)) \subseteq
  \R^d \setminus \supp \rho_R$.

\label{lem:no-large-holes}
\end{lem}

\begin{proof}
  If this is not the case, take $k \in \{1,\dots,d\}$ and ${a_k
    \in \R}$ with ${\pi_k^{-1}([a_k-R_6,a_k+R_6])} \subseteq \R^d
  \setminus \supp \rho_R$, and such that the support of $\rho_R$
  intersects both the ``left'' part $${H_\mathrm{L} :=
    \pi_k^{-1}((-\infty,a_k-R_6])}$$ and the ``right'' part
  $$
{H_\mathrm{R} := \pi_k^{-1}([a_k+R_6,\infty))}.
$$ 
Take $0 <  \epsilon_k \leq R_6$ and $\epsilon = (0,\dots,0,\epsilon_k,0,\dots,0)
  \in \R^d$ with $k$-th coordinate $\epsilon_k$, and consider
  \begin{equation*}
    \tilde{\rho}_R
    := \rho_R \rst_{H_\mathrm{L}}
    + T_\epsilon \big(\rho_R \rst_{H_\mathrm{R}}\big),
  \end{equation*}
  where $\mu \rst_{A}$ denotes the restriction of a measure $\mu$ to a
  set $A$, $T_\epsilon(\mu)$ the push-forward of a measure $\mu$ by the $\epsilon$-translation $
  T_\epsilon\: x \mapsto x - \epsilon$, and as usual $\chi_A$ denotes the characteristic
  function of a set $A$. Clearly $\tilde{\rho}_R \in \P_R(\R^d)$ and
  it is the result of slightly moving to the ``left" the part of
  $\rho_R$ in the $k$-coordinate which is to the ``right" of
  $a_k+R_6$. By Hypotheses \ref{hyp:sym} and \ref{hyp:growth},
  $\tilde{\rho}_R$ has lower energy than $\rho_R$:
  \begin{align*}
    E(\tilde{\rho}_R)
    &=
    E\big( \rho_R \rst_{H_\mathrm{L}} \big) +
    E\big( T_\epsilon \big(\rho_R \rst_{H_\mathrm{R}}\big) \big)
    + \ird \ird W(x-y) \d \rho_R \rst_{H_\mathrm{L}} (x)
    \d T_\epsilon \big(\rho_R \rst_{H_\mathrm{R}}\big)(y)
    \\
    &=
    E\big( \rho_R \rst_{H_\mathrm{L}} \big) +
    E\big( \rho_R \rst_{H_\mathrm{R}} \big)
    + \ird \ird W(x-y-\epsilon)
    \d \rho_R \rst_{H_\mathrm{L}} (x)
    \d \rho_R \rst_{H_\mathrm{R}}(y)
    \\
    &<
    E\big( \rho_R \rst_{H_\mathrm{L}} \big) +
    E\big( \rho_R \rst_{H_\mathrm{R}} \big)
    + \ird \ird W(x-y)
    \d \rho_R \rst_{H_\mathrm{L}} (x)
    \d \rho_R \rst_{H_\mathrm{R}}(y)
    \\
    &= E(\rho_R).
  \end{align*}
  Notice that we use here the translation invariance of the energy
  $E$ and the fact that $$\rho_R = \rho_R \rst_{H_\mathrm{L}} + \rho_R
  \rst_{H_\mathrm{R}}.$$ The strict inequality is due to $W$ being
  strictly increasing in the $k$-th coordinate on $\R^{k-1} \times
  [2R_6-\epsilon_k,\infty) \times \R^{d-k}$ and to our assumption that
  the support of $\rho_R$ intersects both $H_{\mathrm{L}}$ and
  $H_{\mathrm{R}}$. This contradicts the fact that $\rho_R$ is a
  global minimiser on $\P_R(\R^d)$.
\end{proof}

\begin{rem}
	Notice that the proof of Lemma \ref{lem:no-large-holes} still works if instead of Hypothesis \ref{hyp:growth} we suppose the following, less restrictive, growth assumption: for every $k \in \{1,\dots,d\}$ there exists $0 < \epsilon_k \leq R_6$ such that if $x \in \R^d$ is with $x_k \geq R_6$, then $W(x + \epsilon) > W(x)$, where $\epsilon = (0,\dots,0,\epsilon_k,0,\dots,0) \in \R^d$ with $k$-th coordinate $\epsilon_k$. This implies that all our results using Hypothesis \ref{hyp:growth} stay true by relaxing it to this growth assumption (in particular Theorem \ref{thm:main}).
\label{rem:growth-relaxation}
\end{rem}

Next we are able to give a uniform bound on the diameter of a
  minimiser $\rho_R$ on $\P_R(\R^d)$ which is independent of $R$. This already
  contains the main part of the proof of existence of a global
  minimiser, since it could be used, for example, to show the
  tightness of a minimising sequence. It is the main ingredient in
  Lemma \ref{lem:non-compactly-supported-measures}, which is the
  existence part of Theorem \ref{thm:main}.

\begin{lem}
  Assume that the potential $W$ satisfies Hypotheses
  \ref{hyp:bfb}--\ref{hyp:growth}. There exists $K > 0$ (depending
  only on $W$ and $d$) such that for all $R \geq 0$ and global
  minimiser $\rho_R$ of the energy \eqref{eq:energy} on $\P_R(\R^d)$,
  the diameter of the support of $\rho_R$ is bounded by $K$.
  \label{lem:compact-support}
\end{lem}

\begin{proof}
  Let ${S}$ be a radius given by Lemma \ref{lem:instability-finite-R}
  (i.e., such that $W$ is unstable on $\P_{R}(\R^d)$ for all $R \geq
  {S}$). Since ${S}$ depends only on $W$, it is clearly enough to show
  the lemma for $R \geq {S}$.

  Take any $x_0 \in \supp \rho_R$. We recursively define $N+1$ points
  $\{x_0,\dots, x_N\}$, for some $N \geq 0$, as follows:
  \begin{enumerate}
  \item If $\supp \rho_R \setminus \bigcup_{i=0}^{n-1} B(x_i,2r)$,
    where $r$ is the constant in Lemma
    \ref{lem:no-small-isolated-mass} and $n$ is the number
      of already selected points, is not empty, then take any
    $x_n$ in that set.
  \item If the above set is empty, then $x_{n-1}$ is the last term of
    the sequence (i.e., $N = n-1$).
  \end{enumerate}
  We notice that this process must end after at most $\lceil 1/m
  \rceil$ steps; this is, $N + 1 \leq \lceil 1/m \rceil$, where $m$ is
  the constant in Lemma \ref{lem:no-small-isolated-mass} and $\lceil
  \cdot \rceil$ is the ceiling function. The reason for this is that,
  for each $i \in \{0 \dots,N\}$, the ball $B(x_i,r)$ contains at
  least a fixed amount $m$ of mass (see Lemma
  \ref{lem:no-small-isolated-mass}), and this mass is not in any of
  the other balls. Also, it is clear that the support of $\rho_R$ is
  contained in $\bigcup_{i=0}^{N} B(x_i,2r)$.

  We write $x_i = (x_i^{(1)},\dots,x_i^{(d)})$ for all $i \in \{0
  \dots,N\}$, and for any $k \in \{1,\dots,d\}$ we relabel the points
  so that $x_0^{(k)} < \dots < x_N^{(k)}$. Then, if $N > 0$, we have
  \begin{equation*}
    x_{i+1}^{(k)} - x_i^{(k)} \leq 4r + 2R_6
  \end{equation*}
  for all $i \in \{0 \dots,N-1\}$, due to Lemma
  \ref{lem:no-large-holes} (otherwise the support of $\rho_R$ would
  have a gap larger than $2R_6$ in the $k$-th coordinate). From
  this we deduce that
  \begin{equation*}
    x_N^{(k)} - x_0^{(k)} \leq N (4r + 2R_6)
    \leq (\lceil 1/m \rceil - 1) (4r + 2R_6).
  \end{equation*}
  Note that this inequality still holds if $N = 0$, as in this case
  $m$ must be 1. Since $k$ is arbitrary, we have that the diameter of
  the support of $\rho_R$ in each coordinate $k$ is bounded by $4r +
  (\lceil 1/m \rceil - 1) (4r + 2R_6)$. Therefore the diameter of the
  support of $\rho_R$ with respect to the 2-Euclidean norm satisfies
  \begin{equation*}
    \diam \left(\supp \rho_R\right)
    \leq \sqrt{d} (4r + (\lceil 1/m \rceil - 1) (4r + 2R_6)) =: K.
  \end{equation*}
  Note that $K$ does not depend on $R$ or on the choice of $\rho_R$.
\end{proof}

\subsection{Minimisers on the whole set of probability measures}
\label{subsec:all}

We now finish the proof of the existence part of Theorem
\ref{thm:main}:
\begin{lem}
  Assume that the potential $W$ satisfies Hypotheses
  \ref{hyp:bfb}--\ref{hyp:growth}. Then there exists a global
  minimiser for the energy \eqref{eq:energy} on $\P(\R^d)$.
  \label{lem:non-compactly-supported-measures}
\end{lem}

\begin{proof}
  Let $K$ be the bound on the diameter of minimisers on $\P_R(\R^d)$
  for all $R\geq 0$ given by Lemma \ref{lem:compact-support}, and
  consider $\rho'$ a global minimiser on $\P_{K}(\R^d)$. We
    show below that $\rho'$ is in fact a global minimiser in all of
    $\P(\R^d)$.

  Given $\rho \in \P(\R^d)$ with compact support, there exists $R\geq
  0$ such that $\rho\in \P_R(\R^d)$.  Let us take $\rho_R$ a global
  minimiser of $E$ on $\P_R(\R^d)$. Then, we have $E(\rho_R)\leq
  E(\rho)$. Due to translation invariance of $E$, it is clear that
  $E(\rho') \leq E(\rho_R)$ for any $R\geq0$ since the support of
  $\rho$ must have diameter less than $K$, and then it can be
  translated to a measure in $\P_{K}(\R^d)$. Therefore, we conclude
  that $E(\rho') \leq E(\rho)$ for all $\rho \in \P(\R^d)$ with
  compact support.

  We want now to show that $\rho'$ is in fact a global minimiser on
  $\P(\R^d)$. Take any $\rho \in \P(\R^d)$. For $n$ large enough such
  that $M_n := \rho\left(B(0,n)\right) > 0$, let us define the
  sequence $(\rho_n)_n$ by
  \begin{equation}
    \rho_n = \dfrac{1}{M_n} \chi_{B(0,n)} \rho.
  \end{equation}
  Then
  \begin{align*}
    E(\rho_n) -\frac{W_\mathrm{min}}{2}
    &= \dfrac{1}{2} \irdrd (W(x - y)-W_\mathrm{min})
    \d \rho_n(x) \d \rho_n(y)
    \\
    &= \dfrac{1}{2M_n^2} \irdrd \chi_{B(0,n)^2}(x,y)
    (W(x - y)-W_\mathrm{min}) \d \rho(x) \d \rho(y).
  \end{align*}
  Applying the Lebesgue monotone convergence theorem, we get
  \begin{equation*}
    E(\rho_n) \xrightarrow[n\to\infty]{}
    \dfrac{1}{2}\irdrd W(x - y) \d \rho(x) \d \rho(y) = E(\rho).
  \end{equation*}
  Moreover, since $\rho_n \in \P_n(\R^d)$ for all $n$ large enough has compact support, we
  have by above that $E(\rho_n) \geq E(\rho')$. Hence $E(\rho) \geq
  E(\rho')$. Therefore $\rho'$ is a global minimiser on $\P(\R^d)$.
\end{proof}

\subsection{Support compactness of minimisers}
\label{subsec:all-compact-support}

The previous section shows the existence of a compactly
supported global minimiser among all probability measures. However,
this does not exclude existence of a global minimiser without compact
support. Corollary \ref{cor:complement} shows that any global minimiser on
$\P(\R^d)$ is actually compactly supported (and hence, due to Lemma
\ref{lem:compact-support}, has support with diameter less than or
equal to $K$), thus finishing the proof of Theorem
\ref{thm:main}. Its proof is based on very similar reasonings used
throughout Lemma \ref{lem:no-small-isolated-mass}.

\begin{proof}[Proof of Corollary \ref{cor:complement}]
  By Remark \ref{rem:constant-general} and Hypothesis
  \ref{hyp:instability} we have that $W \ast \rho = 2E(\rho)$
  $\rho$-almost everywhere and we can take $A'$ such that $E(\rho) <
  A' < \frac12 W_\infty$. Then, similarly to the proof
  of Lemma \ref{lem:no-small-isolated-mass}, we show that for all $x_0
  \in \supp\rho$,
  \begin{equation*}
    \int_{B(x_0,r'')} \d \rho(x) \geq \dfrac{A' - E(\rho)}{A' - \frac 12 W_\mathrm{min}} =: m' > 0,
  \end{equation*}
  where $r''$ can be found as in the proof of Lemma
  \ref{lem:no-small-isolated-mass}. The result follows immediately from a contradiction argument. Indeed, suppose that $\rho$ is not compactly supported. Then we can choose a sequence of $\lceil1/m'\rceil + 1$ points in its support, where $\lceil\cdot\rceil$ is the ceiling function, such that the balls with centres these points and radii $r''$ do not intersect. By the inequality above this implies that the total mass of $\rho$ is greater than 1, contradicting the fact that $\rho$ is a probability measure.
\end{proof}

\subsection{Corollary for local minimisers}
\label{subsec:corollary-local-minimisers}

Under Hypotheses \ref{hyp:bfb} to \ref{hyp:growth}, Theorem
\ref{thm:main} trivially ensures existence of compactly supported
local minimisers in any topology. However, it is not sufficient to
show that any local minimiser must have compact support. Let us
restrict ourselves to local minimisers with respect to the quadratic Wasserstein
distance $d_2$ (for a definition, see \cite[Section 2]{BCLR2} for example). We know by \cite[Theorem 4(i)]{BCLR2} that, under
Hypotheses \ref{hyp:bfb} to \ref{hyp:sym} and \ref{hyp:lsc}, if $\rho
\in \P(\R^d)$ is a $d_2$-local minimiser with $E(\rho) < +\infty$,
then it satisfies $W \ast \rho = 2E(\rho)$ $\rho$-almost everywhere:
that is, Lemma \ref{lem:constant} (and Remark
\ref{rem:constant-general}) is true for $d_2$-local minimisers on $\P(\R^d)$. We
give here a generalisation of Corollary \ref{cor:complement} by restricting the instability condition of Theorem
\ref{thm:main} to the subset of $d_2$-local minimisers:

\begin{cor}
  Assume that $W\:\R^d \to \R \cup \{+\infty\}$ satisfies Hypotheses
  \ref{hyp:bfb}--\ref{hyp:sym} and \ref{hyp:lsc},
  and is such that $W_\infty := \lim_{|x| \to \infty}
  W(x)$ exists (being possibly equal to $+\infty$). Suppose moreover that there exists a $d_2$-local minimiser $\rho \in \P(\R^d)$ of the energy
  \eqref{eq:energy} with $E(\rho) < \frac12 W_\infty$. Then $\rho$ is
  compactly supported.
\label{cor:local-minimisers}
\end{cor}
\begin{proof}
  The proof is direct by following the arguments of the proof of Corollary \ref{cor:complement} given in Section \ref{subsec:all-compact-support}, and using the result \cite[Theorem
  4(i)]{BCLR2} for $d_2$-local minimisers and the fact that $E(\rho) <
  \frac12 W_\infty$.
\end{proof}

\begin{rem}
	For potentials with $\lim_{|x| \to \infty} W(x) = +\infty$, the instability condition of Corollary \ref{cor:local-minimisers} is automatically verified by any non-trivial $d_2$-local minimiser (that is a local minimiser with finite energy). This is the case, for instance, of the power-law potential given in Proposition \ref{prop:examples}\eqref{item:power} when $a \geq 0$.
\end{rem}

\section{Examples and non-existence of minimisers}
\label{sec:examples-nonexistence}

\subsection{Examples of potentials with minimisers}
\label{subsec:examples}

We want to give explicit examples of potentials $W$ satisfying the
hypotheses of Theorem \ref{thm:main}. To this end we first state a
lemma which gives sufficient conditions for a potential to be
unstable, and therefore for Hypothesis \ref{hyp:instability} to
hold. A similar result can be found in \cite[Section 3.2]{Ruelle},
where alternative conditions are also given for a potential to be
unstable. In the following the subscripts $_+$ and $_-$ stand for
positive and negative part, respectively.
\begin{lem}
  Let $W$ be a potential satisfying Hypotheses
  \ref{hyp:bfb} and \ref{hyp:integrable}, and assume furthermore that $W_\infty :=
  \lim_{|x| \to \infty} W(x)$ exists (being possibly equal to
  $+\infty$).
  \begin{enumerate}[(i)]
  \item \label{item:infinite} If $W_\infty = +\infty$, then $W$ is
    unstable.
  \item \label{item:finite} If $W_\infty < +\infty$, call $\tilde{W}
    := W - W_\infty$. If $\tilde{W}_+$ is
    integrable and $\ird \tilde{W} < 0$ (being possibly equal to $-\infty$),
    then $W$ is unstable.
  \end{enumerate}
  \label{lem:instability-conditions}
\end{lem}

\begin{proof}
\underline{\eqref{item:infinite}.} This case is trivial since, by Hypothesis \ref{hyp:integrable}, any uniform distribution on a given ball has finite energy.
\medskip

\underline{\eqref{item:finite}.} Let us define, for all $R > 0$, the following probability measure:
\bes
	\rho_R = \dfrac{1}{|B(0,R)|} \chi_{B(0,R)}.
\ees
Then compute easily
\begin{align*}
	E(\rho_R) - \frac12 W_\infty &= \dfrac{1}{2|B(0,R)|} \ird \phi_R(x) \tilde W(x) \d x,
\end{align*}
where $\phi_R := \frac{1}{|B(0,R)|}\chi_{B(0,R)} \ast \chi_{B(0,R)}$. Remark that $\phi_R \leq 1$ for all $R > 0$ and $(\phi_R\tilde W)_{R>0}$ converges pointwise to $\tilde W$ on $\R^d$ as $R \to \infty$. Assume first that $\tilde W_-$ is integrable, i.e., by our hypothesis on $\tilde W_+$, $\tilde W$ is integrable. Then, by the Lebesgue dominated convergence theorem we get
\bes
	\ird \phi_R(x)\tilde W(x) \d x \xrightarrow[R \to \infty]{} \ird \tilde W(x) \d x < 0.
\ees
Thus there exists $R$ large enough such that $E(\rho_R) < \frac12 W_\infty$, which shows the results. Now assume that $\tilde W_-$ is not integrable, i.e., $\ird \tilde W = -\infty$. By above we have
\begin{align*}
	E(\rho_R) - \frac12 W_\infty &= \dfrac{1}{2|B(0,R)|} \left(\ird \phi_R(x)\tilde W_+(x) \d x + \ird \phi_R(x) \tilde W_-(x) \d x\right)\\
	&\leq \dfrac{1}{2|B(0,R)|} \left(\ird \tilde W_+(x) \d x + \ird \phi_R(x)\tilde W_-(x) \d x\right).
\end{align*}
Since $(\phi_R\tilde W_-)_{R>0}$ is non-increasing, non-positive and converges pointwise to $\tilde W_-$ on $\R^d$ as $R \to \infty$, the Lebesgue monotone convergence theorem yields
\bes
	\ird \phi_R(x) \tilde W_-(x) \d x \xrightarrow[R \to \infty]{} \ird \tilde W_-(x) \d x = -\infty.
\ees
Therefore, since $\ird \tilde W_+$ is finite, there exists $R$ large
enough such that $E(\rho_R) < \frac12 W_\infty$, which ends the proof.
\end{proof}

In the following proposition we use the result above to find explicit
potentials satisfying all Hypotheses \ref{hyp:bfb} to
\ref{hyp:growth}, and therefore for which Theorem \ref{thm:main} is
applicable.  

\begin{prop}
	Consider the following potentials for all $x \in \R^d$ and $C_A,C_R,\ell_A,\ell_R > 0$:
\begin{enumerate}[(i)]
	\item \label{item:power} \emph{(Power-law potential)} $W(x) = \dfrac{|x|^a}{a} - \dfrac{|x|^b}{b}$ with  $-d < b < a$,
	\item \label{item:morse} \emph{(Morse potential)} $W(x) = C_Re^{-\frac{|x|}{\ell_R}} - C_Ae^{-\frac{|x|}{\ell_A}}$ with $\ell_R < \ell_A$ and $\frac{C_R}{C_A} < \left(\frac{\ell_A}{\ell_R}\right)^d$,
\end{enumerate}
with the convention $\frac{|x|^0}{0} = \log|x|$. Each of these potentials satisfies Hypotheses \ref{hyp:bfb} to \ref{hyp:growth}.
\label{prop:examples}
\end{prop}
\begin{proof}
  \underline{\eqref{item:power}.} Hypotheses \ref{hyp:sym} and
  \ref{hyp:lsc} are trivially respected, as well as Hypotheses
  \ref{hyp:integrable} and \ref{hyp:growth} since $-d < b <
  a$. Furthermore, since $a > b$, we have that $W$ satisfies
  Hypothesis \ref{hyp:bfb}. We are only left to show Hypothesis
  \ref{hyp:instability}. Let us first assume $a \geq 0$. Then $W(x)
  \to +\infty =: W_\infty$ as $|x| \to \infty$. Therefore, by Lemma
  \ref{lem:instability-conditions}\eqref{item:infinite} we have that
  $W$ satisfies Hypothesis \ref{hyp:instability}. On the other hand,
  in the case $a < 0$ we have $W(x) \to 0 =: W_\infty$ as $|x| \to
  \infty$. Since $a > b$, $W$ is asymptotic to $-|x|^b/b$ as $|x| \to
  \infty$, while $W_+$ is integrable since $-d <
  a$. This shows that $\ird W = -\infty$, so Lemma
  \ref{lem:instability-conditions}\eqref{item:finite} applies
  to show that $W$ satisfies Hypothesis \ref{hyp:instability}.
  \medskip

  \underline{\eqref{item:morse}.} Hypotheses \ref{hyp:integrable},
  \ref{hyp:sym} and \ref{hyp:lsc} are trivially respected, as well as
  Hypotheses \ref{hyp:bfb} and \ref{hyp:growth} by our assumptions on
  the parameters. Furthermore $W(x) \to 0 =: W_\infty$ as $|x| \to
  \infty$, and one may check that
  \begin{align*}
    \ird W(x) \d x 
    =
    C'\Gamma(d)(C_R\ell_R^d - C_A\ell_A^d),
  \end{align*}
  where $C' > 0$ is a constant coming from a spherical change of
  variables and $\Gamma$ is the Gamma-function. Therefore $\ird W < 0$
  for the given range of parameters, and by Lemma
  \ref{lem:instability-conditions}\eqref{item:finite} we obtain that
  $W$ satisfies Hypothesis \ref{hyp:instability}.
\end{proof}

\subsection{Non-existence of minimisers}
\label{subsec:sufficient-non-existence}

In the work of Simione, Slep\v{c}ev \& Topaloglu \cite{SST} conditions
are given for the non-existence of minimisers of the interaction
energy. Here, for completeness, we rewrite their result in Theorem \ref{thm:non-existence} adapting it to our hypotheses and using a slightly
different language. Let us consider the following assumption on the potential:
\begin{hyp}
  The limit \eqref{eq:Einf} exists and $W$ is such that there is $\rho
  \in \P(\R^d)$ with $E(\rho) \leq \frac12 W_\infty$.
  \label{hyp:instability-equality}
\end{hyp}
This is a formulation of Hypothesis \ref{hyp:instability} generalised
to the equality case. Our main result, as given in Theorem
\ref{thm:main}, can now be extended to the following:
\begin{thm}
  Suppose that the potential $W$ satisfies Hypotheses
  \ref{hyp:bfb}--\ref{hyp:sym}, \ref{hyp:lsc} and
  \ref{hyp:growth}. Also assume it is such that the limit
  \eqref{eq:Einf} exists and the positive part of $\tilde{W}:= W -
  W_\infty$ is integrable if $W_\infty < +\infty$. Then $E$ admits a
  global minimiser if and only if $W$ satisfies Hypothesis
  \ref{hyp:instability-equality}.
\label{thm:non-existence}
\end{thm}

\begin{proof}
  \underline{Sufficiency.} The sufficiency is almost direct by Theorem
  \ref{thm:main}. However we still need to cover the equality case. By
  contradiction, suppose that there is $\rho \in \P(\R^d)$ with
  $E(\rho) = \frac12 W_\infty$ and that there exists no global minimiser of
  the energy. Then there must be a probability measure $\rho'$ such
  that $E(\rho') < E(\rho) = \frac12 W_\infty$. Now we can apply Theorem
  \ref{thm:main} to get that there exists a global minimiser for $E$,
  which contradicts the non-existence assumption.  \medskip

  \underline{Necessity.} Suppose that there exists a global minimiser
  $\rho' \in \P(\R^d)$ for the energy \eqref{eq:energy}. If
  $W_\infty = +\infty$, then the result is trivial. Assume that
  $W_\infty < +\infty$. By definition of $W_\infty$, we
  have $\lim_{|x| \to \infty}\tilde{W}(x) = 0$. Then, since $W$
  satisfies Hypothesis \ref{hyp:growth}, we know there exists $r > 0$
  large enough such that $\tilde{W}(x) \leq 0$ for all $x \in \R^d$
  such that $|x| > r$. Thus we know that $\ird \tilde{W} \neq +\infty$
  by Hypothesis \ref{hyp:integrable}. Now, since $\tilde{W}_+$ is
  integrable, by proceeding as in the proof of Lemma
  \ref{lem:instability-conditions} and using the same notation, we
  have two cases: either $\ird \tilde W$ is finite and therefore
  $E(\rho_R) \to \frac12 W_\infty$ as $R \to \infty$, or $\ird \tilde W
  = -\infty$ and therefore there is $R$ large enough such that
  $E(\rho_R) < \frac12 W_\infty$. In both cases we get $E(\rho') \leq
  \frac12 W_\infty$. Hence result.
\end{proof}

\begin{rem}
  Theorem \ref{thm:non-existence} shows that, under its
    hypotheses, the only stable potentials
  for which global minimisers exist are the ones such that the
  equality case in the stability definition holds, i.e., the ones such
  that you can find $\rho \in \P(\R^d)$ with $E(\rho) = \frac 12
  W_\infty$. For radially symmetric potentials this
  is also true without Hypothesis \ref{hyp:growth}, as proven in
  \cite{SST}. An example of such potential is the following: \bes W(x)
  = |x|^2 e^{-|x|^2}, \qquad x \in \R^d.  \ees Indeed $W$ is
    radially symmetric and satisfies all the hypotheses of
    Theorem \ref{thm:non-existence} but Hypothesis \ref{hyp:growth}, and is stable with obviously $E(\delta_0) =  \frac 12 W(0) = 0 = \frac 12 W_\infty$, where $\delta_0$ is the
  Dirac measure centred at the origin.
\end{rem}

\subsection*{Acknowledgements}

The authors would like to thank D. Chafa\"{\i} for mentioning some
references regarding the circular law in probability. J.~A. Ca\~nizo
acknowledges support from projects MTM2011-27739-C04-02
and the Marie-Curie CIG project
KineticCF. J.~A. Carrillo acknowledges support from projects
MTM2011-27739-C04-02, the Royal Society through a Wolfson Research Merit
Award, and the Engineering and Physical Sciences Research Council (UK)
grant number EP/K008404/1.


\end{document}